\documentclass{article}
\usepackage{amsthm,amsmath,amsfonts,amssymb,graphics} 

\usepackage[dvips]{graphicx}
\usepackage{verbatim}
\title{Extensions of Billingsley's Theorem via Multi-Intensities}
\author{Richard Arratia  \\ Fred Kochman \\ Victor S. Miller}
\date{June 2012}

\def\e{\mathbb{E} \,}

\def\p{\mathbb{P}}

\def\P{\mathbb{P}}

\def\eps{\varepsilon}
\def\diam{{\rm diam}\,}
\def\vol{{\rm vol}\,}
\def\om{{\omega}}
\def\Om{{\Omega}}

\newcommand{\cS}{{\cal S}}

\newcommand{\pp}{{\frak{p}}}
\newcommand{\qq}{{\frak{q}}}

\newcommand{\ignore}[1]{}

\newcommand{\gem}{GEM}
\newtheorem{theorem}{Theorem}
\newtheorem{lemma}{Lemma}
\newtheorem{proposition}{Proposition}
\newtheorem{corollary}{Corollary}

\begin{document}

\maketitle

\begin{abstract} 
Let $p_1 \ge p_2 \ge \dots$ be the prime factors of a
random integer chosen uniformly from $1$ to $n$, and let
$$ 
\frac{\log p_1}{\log n}, \frac{\log p_2}{\log n}, \dots
$$
be the sequence of scaled log factors.
Billingsley's Theorem (1972), in its modern formulation, asserts that
the limiting  
process, as $n \to \infty$, is the Poisson-Dirichlet process with parameter $\theta =1$.
 
In this paper we give a new proof, inspired by the 1993 proof by
Donnelly and Grimmett, and extend the result to factorizations of
elements of normed arithmetic semigroups satisfying certain growth
conditions, for which the limiting Poisson-Dirichlet process need not
have $\theta =1$.  We also establish Poisson-Dirichlet limits, with
$\theta \ne 1$, for ordinary integers conditional on the number of
prime factors deviating from the usual value $\log \log n$.

At the core of our argument is a purely probabilistic lemma
giving a new criterion for convergence in distribution to a
Poisson-Dirichlet process, from which the number-theoretic
applications follow as straightforward corollaries. The lemma uses
ingredients similar to those employed by Donnelly and Grimmett, but
reorganized so as to allow subsequent number theory input to be
processed as rapidly as possible.

A by-product of this work is a new characterization of
Poisson-Dirichlet processes in terms of multi-intensities.
\end{abstract}

\section{Introduction}

In this paper we provide a new criterion for convergence
in distribution to PD($\theta$) --- the Poisson-Dirichlet
process with parameter $\theta$ --- and then apply it to extend
Billingsley's theorem on the asymptotic distribution of log prime
factors of a random integer to much more general number theoretic
contexts.

The new criterion is an application of an existing general weak
convergence lemma from \cite{AK}, restated here as Proposition
\ref{proposition 1}, which supplements Alexandrov's Portmanteau
Theorem on equivalent conditions for weak
convergence~\cite{BillingsleyWeak2}.  That lemma governs, in
particular, the convergence of a sequence of discrete nonlattice
random variables to a continuum limit possessing a smooth density.
Here, we adapt it for direct application to a sequence of random
multisubsets of (0,1], with a hypothesis yielding the limiting
PD($\theta$).

The following is the simplest version of our new criterion:
Given a sequence $A_n$ of multisubsets of $(0,1]$, let
  $T_n$ denote the sum of the elements of $A_n$, counting
  multiplicities,
and for any set $S \subset (0,1]$ let $|A \cap S|$ denote the cardinality of the 
intersection, also counting multiplicities. 
Also, let $L(n)=(L_1(n),L_2(n),\dots)$, where $L_i(n) := $ the $i^{\rm th}$
largest
element of $A_n$ if $i \le |A_n|$, and   $L_i(n) := 0$ if $i >
|A_n|$.

Lemma~\ref{maintheta1} then asserts the following: 
\emph{Suppose that
$T_n \le 1$ almost surely, for all $n$, and that for any collection of
disjoint closed intervals $I_i = [a_i,b_i] \subset (0,1], i =
  1,\dots,k$ satisfying $b_1 + \cdots + b_k <1$, for any $k \ge 1$, 
we have 
\begin{equation}\label{simple}
\liminf \e |A_n \cap I_1|\cdots |A_n \cap I_k| \ge 
\prod_{i=1}^k  \log\frac{b_i}{a_i} 
\end{equation}
as $n \to \infty.$
Then $L(n)$ converges in distribution to a ${PD}(1)$.
}

The more general version, Lemma~\ref{maintheta}, allows a limiting
PD($\theta$) with $\theta \ne 1$, and has a somewhat more complicated
expression on the right-hand side of the inequality. Nonetheless both
versions are easy to use in our applications.

In proving Billingsley's original theorem, for instance, the multiset
$A_n$ appearing above consists of the log prime factors, $\log_n p$, of a random 
integer in $[1,n]$; and so each number
$|A_n \cap I_i|$ is simply the number of prime factors, counting multiplicities, 
falling into $[n^{a_i},n^{b_i}]$. 
The main step of the proof, 
confirmation of the hypothesis \eqref{simple}, 
reduces to scarcely more than a citation of Mertens' formula \cite{ingham}
\begin{equation}\label{Mertensoriginal}
\sum_{p \le x} \frac{1}{p} = \log \log x + B +o(1),
\end{equation}
where $B$ is a constant.
 
In the later sections of this paper, we apply our new criterion to give
\begin{itemize}
\item A reproof of the original Billingsley's theorem;
\item a generalization to a class of normed arithmetic semigroups
for which an analogue of Landau's prime ideal theorem is valid, still
yielding a PD(1) limit;
\item a generalization to a class of normed arithmetic semigroups
satisfying the growth hypotheses of Bredikhin's theorem (\cite{postnikov},
Section 2.5), yielding PD($\theta$) limits with $\theta \ne 1$; and
\item a final generalization to ordinary integers, conditional
on the number of prime factors in the selected integer deviating unusually from the
prescription of Tur{\'a}n's theorem; here too the limiting
PD($\theta$) has $\theta \ne 1$.
\end{itemize}

 Our work was inspired by the proof of Billingsley's theorem
given by Donnelly and Grimmett \cite{DonnellyGrimmett}, whose
ingredients are included amongst our own.  Our initial motivation
was to place as much of the burden as possible on self-contained
probability tools and isolate the use of number theoretic input. One
difference in our approach is that, internal to the new convergence
lemma, we work directly with the density function of the \gem\
distribution instead of aiming for a limit process of the component
uniforms, as they do.

\section{Probability Background}

\subsection{PD($\theta$) and GEM($\theta$)}\label{intro1}
 For our present purposes the 
{\em Poisson-Dirichlet point process $PD(\theta)$ with parameter $\theta>0$} can
be characterized in the following two equivalent ways:
\begin{itemize}
\item[  I)  ]PD($\theta$)  is  the Poisson point
process with scale invariant intensity measure $\theta\, dx/x$ on $(0,1)$, conditioned on the sum of
the arrivals being 1;
\item[ II)  ]
Let $U_1, U_2,\dots$ be a sequence of independent uniforms on $(0,1)$, let
$V_i = U_i^{1/\theta}$ for $i\ge1$, and let $G_1 = 1-V_1, G_2 =
V_1(1-V_2),$ \mbox{$ G_3 = V_1 V_2(1-V_3), \dots$.}
 Let $L_1 \ge L_2\ge L_3 \ge \dots$ be the outcome of sorting $G_1, G_2, G_3,\dots$ into 
descending order, or {\em ranking} them, as we will say.  Then $L_1, L_2, L_3,\dots$ are the 
arrivals of the PD($\theta$) in $(0,1)$.
\end{itemize}
For information on this process and its role in combinatorial modeling, as well as alternative
characterizations with proofs of equivalence
see, e.g., \cite{abtbook,ABTarticle,kingman}.

 The GEM($\theta$) \emph{ process} 
$$
G = (G_1,G_2,\dots),
$$ appearing in the second characterization, has itself been well
studied (see, e.g.,\cite{abtbook}). It is easy to see that with probability one we have
$G_1+ \cdots + G_k < 1$, for all $k$. So, in particular, there are no
positive accumulation points and hence ranking $G$ is actually
possible.
 
 We will exploit the following result from 
\cite{joyce};  see also \cite{BillingsleyWeak2}, p. 42--43. 
\begin{lemma}\label{joyce}
 Let 
$$
G(n) = (G_1(n),G_2(n),\dots)
$$
be a sequence of processes of nonnegative numbers, each with almost surely finite sum; 
and for each $n$ let
$$
L(n) = (L_1(n),L_2(n), \dots)
$$ 
be the ranked version of $G(n)$.
Suppose $G(n)$ converges in distribution to \gem($\theta$).   Then
$L(n)$ converges
in distribution to PD($\theta$).
\end{lemma}

 For each $k,$ the first $k$ coordinates $G_1,\dots, G_k$ of the \gem($\theta$) 
possess a joint probability 
density function $f_{\theta}$, with the formula (see 
(5.28) in Section~5.4 of~\cite{abtbook})
\begin{equation}\label{fdds G}
  f_{\theta}(x_1,\dots,x_k) = 
\frac{\theta^k(1-x_1 - \cdots -x_k)^{\theta-1}}
     {1(1-x_1)(1-x_1-x_2)\cdots (1-x_1-\cdots-x_{k-1})} 
\end{equation}
for $ (x_1,x_2,\dots,x_k)   \in U$,
where $U \subset \mathbb{R}^k$  is the open set
\begin{equation}\label{def U G}
 U := \{  x \in (0,1)^k: 0 <x_1 + \dots + x_k < 1 \  \mbox{ and } x_i \neq x_j
\mbox{ for } 1 \le i < j \le k  \}.
\end{equation}
While exclusion of the subdiagonals is not usually imposed in the
definition of $U$, excluding a set of Lebesgue measure $0$ does no
harm, and it definitely finesses a technical issue arising in the
proof of our main lemma.

\subsection{A Proposition on Weak Convergence}

 For the reader's convenience, we quote Proposition 2.1 from
\cite{AK}, upon which the new result will depend.  Here we let
$X$ be a random element of $\mathbb{R}^k$ with density $f$ of the form
$f = f_U 1_U$, where $U \subset \mathbb{R}^k$ is open, the function
$f_U : U \to (0,\infty)$ is continuous, and $1_U: \mathbb{R}^k \to
\{0,1\}$ is the indicator function of $U$.

\begin{proposition}\label{proposition 1}
 Let $X$ be defined as above, and 
let $X_n$, $n=1,2,\dots$, be arbitrary random elements of $\mathbb{R}^k$.

 Suppose that, for every $\eps > 0$,  there exists $R < \infty$ for which
every  closed coordinate box $B \subset U \mbox{  satisfying }$
\begin{equation}\label{lemma hyp 0}
\diam(B) < d(B, U^c)/R 
\end{equation}
also satisfies
\begin{equation}\label{lemma hyp 1}
  \liminf_n \p(X_n \in B) \ge (1-\eps)\  \vol(B)\  \inf_B f.
\end{equation}
Then $X_n \Rightarrow X$. That is, as $n \to \infty$,
$X_n$ converges in distribution to $X$.
\end{proposition}

\section{Convergence to PD($\theta$)}

\subsection{Preliminaries}

 Our convergence criterion is most conveniently cast in the language
of random multisubsets\footnote{ By saying $A$ is a multisubset of
$U$, we mean that $U$ is a universal set, and for each $u \in U$, the
multiplicity $m_u=m_u(A)$ of $u$ as an element of $A$ is a nonnegative
integer.  There is of course an alternate reading of the phrase, with
``$A$ is a multisubset of $B$'' to mean that both $A,B$ are multisets,
and for each $u$ in the underlying universal set, $m_u(A) \le
m_u(B)$.} of the interval $(0,1]$.  In the course of the proof we will
also need to consider size-biased permutations.

 We consider only multisets whose multiplicities are all finite.

 Informally, a process that generates random countable (or
finite) multisets $A\subset (0,1]$ has been fully specified provided
that for any finite collection $S_1,\dots,S_k$ of Borel subsets of
$(0,1]$, the cardinalities $|A \cap S_1|,\dots,|A\cap S_k|$, including
multiplicities, have well-defined joint probability
distributions.\footnote{ This induces a probability measure on the
space whose points are countable subsets of $(0,1]$. For further
information, including the identification of random multisets with
random $\sigma$-finite integer-valued measures on the ambient space,
see, e.g., \cite{kallenberg}, Chapter 12.} Though infinite
cardinalities may occur, for any singleton $\{ x \} \subset (0,1]$ we
must have $|A \cap \{x \} | < \infty$.  The joint distributions must
obey any constraints implied by set inclusions.

 Given an at most countable  fixed multiset $A$ of numbers in $(0,1]$ 
(or, indeed, lying anywhere in the positive reals)
with finite sum $t$
(where each summand is included according to its multiplicity), a \emph{size-biased permutation} is an ordered list generated by the following
process: The first element selected $\sigma_1$ equals $s$ with probability proportional to 
$m_ss$ where $m_s$ is the multiplicity of $s$ in $A$; explicitly,
$\p(\sigma_1=s)=m_s s/t$. Thereafter, 
conditional on selections already made, for
any element $s$ remaining in $A$ the next element selected is $s$ with probability proportional 
to $m_s's$, where $m_s'$ is the multiplicity of $s$ among those 
elements yet remaining to be selected. If $|A|<\infty$, we explicitly
set $\sigma_k=0$ for all $k > |A|$. (For multisets the count $|A|$ includes multiplicities.) 

 We may also take size-biased permutations of random multisets,
with sum $T<\infty$: The probability $P(\sigma_1 = s_1,\dots, \sigma_k
= s_k)$, say, that the first $k$ selections are $s_1,\dots,s_k,$ is
calculated by first conditioning on the random multiset $A$ and
calculating $P(\sigma_1 = s_1,\dots, \sigma_k = s_k|A)$ recursively,
as above, and then taking the expectation as $A$
varies.\footnote{ The probability distribution of $P(\sigma_1 =
s_1,\dots, \sigma_k = s_k|A)$ is itself determined by the joint
distributions of cardinalities of intersections, together with the sum
$T$: The occurrence or non-occurrence of the event \{$P(\sigma_1 =
s_1,\dots, \sigma_k = s_k|A) < x\}$ is determined by the
multiplicities of $s_1,\dots,s_k$ and the sum of all remaining
elements, taken with multiplicities; so the probability of that event
is a function of the joint distribution of those quantities, i.e., the
joint distribution of the intersection cardinalities $|A \cap
\{s_1\}|, \dots,|A \cap \{s_k\}|$, and $T$.  So the expectation of
$P(\sigma_1 = s_1,\dots, \sigma_k = s_k|A)$ is taken, by definition,
with respect to this latter joint distribution.}

\subsection{The Main Lemma, $\theta =1$}\label{sectheta1}

 Since the special case $\theta =1$ is a bit less complicated than the
general case, yet already suffices for the classical version of
Billingsley's result as well as for our first extension, we state and
prove the result for this case first.

 Given an arbitrary sequence $A_n$ of random multisubsets $A_n$ of
$(0,1]$, for each $n$ define $L(n)$ to be the sequence of elements of
$A_n$, including multiple occurrences, ranked by decreasing size, and
padded with an infinite string of $0$'s if $A_n$ is finite. That is,
we let $L(n)=(L_1(n),L_2(n),\dots)$ where $L_i(n) := $ the $i^{\rm
th}$ largest element of $A_n$ if $i \le |A_n|$, and $L_i(n) := 0$ if
$i > |A_n|$.  Also, define $ T_n := \sum_{a \in A_n} a$, where the
defining sum is taken with multiplicities.

\begin{lemma}\label{maintheta1}
 Given an arbitrary sequence $A_1, A_2,\dots$ of random multisubsets of
$(0,1]$, with 
associated ranked sequences $L(1),L(2),\dots$ of elements,
assume the following:  first,
\begin{equation}\label{k hyp 2}
 \p( T_n \ \le 1) =1, \end{equation} 
and second, for for any collection of disjoint closed $I_i = [a_i,b_i]
\subset (0,1], \ i=1,\dots,k$ satisfying the hypothesis
\begin{equation}\label{k hyp 1}
b_1+\cdots + b_k < 1,
\end{equation}
we have 
\begin{equation}\label{intensineq}
 \liminf_{n \to \infty}  \e \prod_{i=1}^k  | A_n \cap I_i|  \ \ge
\prod_{i=1}^k \frac{b_i-a_i}{b_i}.
\end{equation}
Then $L(n) \Rightarrow (L_1,L_2,\dots)_{\theta =1}$, the Poisson-Dirichlet
distribution with \mbox{$\theta =1$}.

If in place of \eqref{intensineq} we assume
\begin{equation}\label{logintensineq}
 \liminf_{n \to \infty}  \e \prod_{i=1}^k  | A_n \cap I_i|  \ \ge
\prod_{i=1}^k \log\frac{b_i}{a_i},
\end{equation}
then the same conclusion holds.
\end{lemma} 
\begin{proof}

Since for $b \ge a >0$ we have $\log(b/a) \ge (b-a)/a$, it suffices to consider only 
\eqref{intensineq}. 
Taking our cue from Donnelly and Grimmett \cite{DonnellyGrimmett}, for each $n$ define a process, 
$$
G(n) = (G_1(n),G_2(n), \dots),
$$
whose components are the successive elements of a size biased permutation of $A_n$, padded with 
zeros if $A_n$ is finite.
We will use Lemma~\ref{proposition 1} with $X$ equal to the first $k$ coordinates of the \gem(1),
in conjunction with \eqref{fdds G} and \eqref{def U G}, 
to show that as $n \to \infty$, the first $k$ coordinates of $G(n)$ converge in distribution to
the first $k$ coordinates of a \gem($1$), for each $k$. Since this implies that $G(n)$ converges
to a \gem($1$), we will then conclude by Lemma~\ref{joyce} that $L(n)$, the ranked 
version of $G(n)$, converges to a PD($1$).

 So let $B=\prod_{i=1}^k [a_i,b_i]$ be a coordinate box whose component intervals satisfy 
our hypotheses. Conditional on $A_n,$ we see that
$$
\P(G_1(n) \in I_1) = \frac{\sum_{a \in A_n \cap I_1}a}{T_n}   \ge |A_n \cap I_1|a_1.
$$

 Conditional also on the first $j$ selections lying in $I_1,\dots,I_j$,
respectively, since their sum must be at least $a_1+ \cdots +a_j,$ the
conditional probability that
$$
G_{j+1} \in I_{j+1}
$$ 
is at least 
$$
\frac{| A_n \cap I_{j+1}|a_{j+1}}{T_n-(a_1+ \cdots + a_j)} \ge 
\frac{| A_n \cap I_{j+1}|a_{j+1}}{1-(a_1+ \cdots + a_j)},
$$
where we have used the disjointness of the intervals to infer that all of $A_n \cap I_{j+1}$
remains available.

 Hence, we find that 
\begin{equation}\label{bias and intensity}
\p ( (G_1(n),\dots,G_k(n)) \in B) \ge 
\end{equation}
$$
\e \left(
 \frac{| A_n \cap I_1|a_1}{1} \ \frac{| A_n \cap I_2|a_2}{1-a_1} \cdots 
\frac{| A_n \cap I_k|a_k}{1-(a_1+ \cdots + a_{k-1})} \right).
$$
Combining hypothesis \eqref{intensineq} with the fact that
$\vol B = \prod (b_i-a_i)$, and using formula \eqref{fdds G} with $\theta =1$,
we see that
$$
\liminf_{n \to \infty} \p ( (G_1(n),\dots,G_k(n)) \in B) \ge 
\prod_{i=1}^k \frac{a_i}{b_i} \ \ \vol(B) \  f_1(a_1,\dots,a_k)
$$ 
$$
 \ge 
\prod_{i=1}^k \frac{a_i}{b_i} \ \ \vol(B) \ \inf_B f_1.
$$ 

 To apply Proposition \ref{proposition 1}, given $\epsilon >0$ it
will suffice to find $R$ large enough so that for all coordinate boxes
satisfying \eqref{lemma hyp 0}, we have $\prod_{i=1}^k \frac{a_i}{b_i}
\ge 1- \epsilon$.  Then \eqref{intensineq} will imply \eqref{lemma hyp
1}. Note that while nothing in the statement of Lemma \ref{proposition
1} explicitly allows us to restrict attention to boxes whose defining
intervals are disjoint, as required for the invocation of
\eqref{intensineq}, our crafty choice of domain $U$ in \eqref{def U G}
makes that automatic. Any closed box lying in $U$ also satisfies
$b_1+\dots+b_k <1$, the other requirement.

 Without harm we may restrict to $\epsilon <1$. Since given any $R>0$, a box $B$ 
satisfying \eqref{lemma hyp 0} satisfies  
\begin{equation}\label{lemmaineq}
b_i-a_i \le \frac{\diam(B)}{\sqrt{k}} < \frac{1}{R \sqrt{k}} \, d(B,U^c)
\le \frac{1}{R \sqrt{k}} \,  a_i,
\end{equation}
 it suffices to take $$R =
 \left[\sqrt{k}\left((1-\eps)^{-1/k}-1\right)\right]^{-1},$$ to get
 $a_i/b_i 
>(1-\eps)^{1/k}$ and hence $\prod a_i/b_i > (1-\eps)$. 

 With this choice of $R$ we have satisfied \eqref{lemma hyp 1}, so
Lemma~\ref{proposition 1} applies; and by the discussion beginning the
proof we are done.
\end{proof}

\subsection{Characterization of PD(1) via Multi-intensity}
\label{sect multi 1}

 Lemma~\ref{maintheta1} above gives a sufficient condition for
convergence to the Poisson-Dirichlet distribution, with parameter
$\theta=1$.  We now explain how this gives a new characterization of
the PD distribution.

 A standard concept for point processes is the \emph{intensity
measure}; in our setup with a random multiset set $A \subset (0,1]$
this is the deterministic measure $\nu$ on the Borel subsets, defined
by $\nu(S)=\e |A \cap S|$.  A standard result in measure theory, the
$\pi-\lambda$ theorem, implies that $\nu$ is determined by its values
on closed intervals, $\nu(I)$ for $I=[a,b]$, $0 <a <b \le 1$.  At this
level, both the Poisson-Dirichlet process PD(1) and the scale
invariant Poisson process with intensity $dx/x$ on $(0,1]$, have the
same intensity with $\nu([a,b]) = \int_a^b dx/x = \ln b/a$.

 Second-order intensity has been considered, for example in
\cite{Brillinger76}.  It is natural to generalize, and define
\emph{multi-intensity} or $k$-fold intensity for $k=1,2,\dots$, by
taking arbitrary choices of $k$ disjoint closed intervals $I_i =
[a_i,b_i]$, setting $B = \prod_{i=1}^k [a_i,b_i]$, and defining
$$
\mu(B) = \e |A \cap I_1| \dots |A \cap I_k|.
$$
In case there is a function $f_k$ on $(0,1]^k \setminus 
\cup_{1 \le i < j \le k} \{x_i=x_j\}$, such that $\mu(B) =
\int_B f_k(x_1,\dots,x_k) \ dx_1 \cdots dx_k$, we say that the random
set $A$ has multi-intensity \emph{density} $f_k(x_1,\dots,x_k)$ 
at $(x_1,\dots,x_k)$. For example, any Poisson process with intensity
$f(x)\, dx$
has multi-intensity density $f_k(x_1,\dots,x_k)
 = f(x_1)
 \times \cdots \times f(x_k) $.  In particular, the
scale invariant Poisson process with intensity $dx/x$ on $(0,1]$ has
multi-intensity density
\begin{equation}\label{PP 1 multi intensity}
{\rm PP\!:\ } f_k(x_1,\dots,x_k) = \frac{1}{x_1 \cdots x_k},
\end{equation}
for all choices of distinct $x_1,x_2,\dots,x_k \in (0,1]$.

 The multi-intensity for the Poisson-Dirichlet is easily derived from
I) in Section~\ref{intro1}, the characterization of PD as PP
conditional on $T=1$, where $T$ is the sum of all the points of the
Poisson process with intensity $dx/x$ on $(0,1]$.  A special
  simplification arises from the property
that the density $p(t)$ for $T$, given explicitly by $p(t)=
e^{-\gamma} \rho(t)$ where $\rho(\cdot)$ is Dickman's function,  satisfies $p(u)/p(1)=1$ for all $u \in (0,1]$.
For distinct $x_1,\dots,x_k \in (0,1]$ with $t := x_1+\cdots+x_k$,
by conditioning the Poisson process on having $T=1$ 
we have for the PD
\begin{equation}\label{PD 1 multi intensity}
{\rm PD\!:\ }
f_k(x_1,\dots,x_k)  = \frac{1}{x_1 \cdots x_k} \ \frac{p(1-t)}{p(1)} 
 = \left\{
\begin{array}{ccc}
    \frac{1}{x_1 \cdots x_k}& \mbox{ if } & t<1 \\
    0                       & \mbox{ if } & t>1
\end{array}
\right.   .
\end{equation}

 To summarize, by comparing \eqref{PP 1 multi intensity} with
\eqref{PD 1 multi intensity}, we see that for $k \ge 2$, the Poisson process and the
Poisson-Dirichlet don't have the same multi-intensity densities, but
their densities agree, when restricted to $(x_1,\dots,x_k)$ with  $t
:= x_1+\cdots+x_k < 1$.

\begin{corollary}\label{cor PD 1}
 View the Poisson-Dirichlet process $(X_1,X_2,\dots)$ as the random
multisubset of $(0,1]$ given by $A = \{X_1,X_2,\dots \}$\footnote{The
condition $T < \infty$ implies that the multiset $A$ can be
reconstructed from the sequence $(X_1,X_2,\dots)$.}. Then the PD is
the unique random $A$ for which both
$$ 
   T := \sum_{x \in A} x\   \mbox{has }  \p(T \le 1 ) = 1,
$$ where the sum is taken with multiplicities, and for each
$k=1,2,\dots$, the multi-intensity density of $A$ is given by the
right side of \eqref{PD 1 multi intensity}.
\end{corollary}
\begin{proof}
 If a multiset $A$ satisfies the given hypotheses, then we can
apply
Lemma~\ref{maintheta1} with $A_n=A$,  for each
$n=1,2,\dots$.   Conversely, we have already noted that starting with
the PD, we have $1=\p(T=1)$, and multi-intensity density as given by  
\eqref{PD 1 multi intensity}.
\end{proof}

\subsection{The Main Lemma, Arbitrary $\theta > 0$}\label{sectheta}

 We keep the notation of Section~\ref{sectheta1}.
\begin{lemma}\label{maintheta}
 Let $\theta >0$.  Given an arbitrary sequence $A_1, A_2,\dots $
of random multisubsets of $(0,1]$, with the associated ranked
sequences $L(1),L(2),\dots$ of elements, make the following
assumption: Suppose that for some $-\infty < \alpha, \beta < \infty$ with $\alpha + \beta
=1 - \theta,$ it is the case that for any collection of disjoint
closed $I_i = [a_i,b_i] \subset (0,1], \ i=1,\dots,k$ satisfying the
hypothesis
\begin{equation}\label{theta k hyp 1}
b_1+\cdots + b_k < 1,
\end{equation}
we have both
\begin{multline}\label{theta intensineq}
 \liminf_{n \to \infty}  \e \prod_{i=1}^k  | A_n \cap I_i|  \ \ge\\
\frac{\theta^k}{(1-a_1-\cdots-a_k)^{\alpha}(1-b_1-\cdots-b_k)^{\beta}}\prod_{i=1}^k \frac{b_i-a_i}{b_i}
\end{multline}
and
\begin{equation}\label{theta k hyp 2}
 \p( T_n \ \le 1) =1. 
\end{equation}
Then
 $L(n) \Rightarrow (L_1,L_2,\dots)_{\theta}$, the Poisson-Dirichlet
process with \mbox{parameter $\theta$.}

If in place of \eqref{theta intensineq} we assume
\begin{multline}\label{logtheta intensineq}
 \liminf_{n \to \infty}  \e \prod_{i=1}^k  | A_n \cap I_i|  \ \ge\\
\frac{\theta^k}{(1-a_1-\cdots-a_k)^{\alpha}(1-b_1-\cdots-b_k)^{\beta}}\prod_{i=1}^k \log\frac{b_i}{a_i},
\end{multline}
then the same conclusion holds.
\end{lemma}

\begin{proof}
 As in the proof of Lemma~\ref{maintheta1}, we appeal to Lemma
\ref{proposition 1}, this time using \eqref{fdds G} with $\theta>0$ to
specify the target limit density. Also, it suffices to consider only
\eqref{theta intensineq} and not \eqref{logtheta intensineq}.
If the coordinates of $G(n) = (G_1(n),G_2(n),\dots)$
are generated as a size-biased permutation of the elements of $A_n$,
padded with zeros if necessary, then \eqref{bias and intensity} 
gives a lower bound on $\p ( (G_1(n),\dots,G_k(n)) \in B)$.  This
combines with hypothesis \eqref{theta intensineq} and formula
\eqref{fdds G} to give
$$
\liminf_{n \to \infty} \p ( (G_1(n),\dots,G_k(n)) \in B) 
$$
$$
\ge \frac{(1-a_1-\cdots-a_k)^{1- \theta}}{(1-a_1-\cdots-a_k)^{\alpha}(1-b_1-\cdots-b_k)^{\beta}}
\prod_{i=1}^k \frac{a_i}{b_i} \ \ \vol(B) \ f_{\theta}(a_1,\dots,a_k)
$$
$$
 =  \left(\frac{1-b_1-\cdots-b_k}{1-a_1-\cdots-a_k}\right)^{-\beta}
\prod_{i=1}^k \frac{a_i}{b_i} \ \ \vol(B) \ f_{\theta}(a_1,\dots,a_k)
$$
$$
 \ge \left(\frac{1-b_1-\cdots-b_k}{1-a_1-\cdots-a_k}\right)^{-\beta}
\prod_{i=1}^k \frac{a_i}{b_i} \ \ \vol(B) \ \inf_B f_{\theta}
$$
$$
\ge \left(\frac{1-b_1-\cdots-b_k}{1-a_1-\cdots-a_k}\right)^{\gamma}
\prod_{i=1}^k \frac{a_i}{b_i} \ \ \vol(B)  \ \inf_B f_{\theta}
$$
where we have written $\gamma$ for $\max(0,-\beta)$.

We now show the preliminary factors can be replaced with $1-\epsilon$:
Given $1 > \epsilon' >0$, 
setting $$R = R_1 =
\left[\sqrt{k}\left((1-\eps')^{-1/k}-1\right)\right]^{-1}$$ will serve, via
\eqref{lemmaineq} 
for boxes complying
with \eqref{lemma hyp 0}, to ensure
that
$$ 
\prod_{i=1}^k \frac{a_i}{b_i} \ge 1- \eps'.
$$ 
As for bounding
$$ 
\left(\frac{1-b_1-\cdots-b_k}{1-a_1-\cdots-a_k}\right)^{\gamma}
$$
when  $\gamma >0$,
complying with \eqref{lemma hyp 0} for a given $R$ also means
$$
b_i-a_i \le \frac{\diam(B)}{\sqrt{k}} < \frac{1}{R \sqrt{k}} d(B,U^c)
\le \frac{1-b_1-\cdots-b_k}{kR}
$$
for each $i$,
where in the rightmost member we have measured distance from the hyperplane $x_1 + \cdots +x_k =1$.
Since $(1- a_1-\cdots-a_k) - (1-b_1-\cdots-b_k) = (b_1 - a_1) + \cdots +(b_k - a_k)$
we get 
$$
(1- a_1-\cdots-a_k) - (1-b_1-\cdots-b_k) \le  \frac{1-b_1-\cdots-b_k}{R},
$$
and so for $R = R_2$ sufficiently large we have 
$$
\left(\frac{1-b_1-\cdots-b_k}{1-a_1-\cdots-a_k}\right)^{\gamma} \ge 1-\epsilon'
$$
as well.  Thus when $\gamma >0$, given $\epsilon >0$ pick $\epsilon'$ small enough so that  
$(1-\epsilon')^2 \ge 1-\epsilon$, and then choose $R$ to be the larger of $R_1$ and $R_2$.
Proposition \ref{proposition 1} now applies, completing the argument.
\end{proof}

\subsection{Characterization of PD($\theta$) via Multi-intensity}
\label{sect multi theta}

 We now
treat the situation for general $\theta>0$, thereby extending the results of 
Section~\ref{sect multi 1}.  We will be brief and
highlight only the changes.

 The scale invariant Poisson process with intensity $\theta\,
dx/x$ on $(0,1]$ has multi-intensity density
\begin{equation}\label{PP theta multi intensity}
{\rm PP\!:\ } f_k(x_1,\dots,x_k) = \frac{\theta^k}{x_1 \cdots x_k},
\end{equation}
for all choices of distinct $x_1,x_2,\dots,x_k \in (0,1]$.
For any $\theta>0$, the density $p(t)$ for $T$, \emph{restricted to}
(0,1], is given (see for example \cite{abtbook}, formula
  (4.20)) by 
$$
  p(t)= \frac{e^{-\gamma \theta}
 t^{\theta-1}}{\Gamma(\theta)}\  \mbox{ for } 0 < t \le 1.
$$
Hence, by conditioning the Poisson $\theta \, dx/x$ process on the event $T=1$,
for distinct $x_1,\dots,x_k \in (0,1]$ with $t := x_1+\cdots+x_k$,
we have multi-intensity density
\begin{equation}\label{PD theta multi intensity}
{\rm PD\!:\ }
f_k(x_1,\dots,x_k)  = \frac{\theta^k}{x_1 \cdots x_k} \ \frac{p(1-t)}{p(1)} 
 = \left\{
\begin{array}{ccc}
    \frac{\theta^k (1-t)^{\theta-1}}{x_1 \cdots x_k}& \mbox{ if } & t<1 \\
    0                       & \mbox{ if } & t>1
\end{array}
\right.   .
\end{equation}

 For $k=1$, this gives the intensity measure of the PD, 
$\nu(dx) = \theta (1-x)^{\theta-1}\, dx/x$ on $(0,1]$, which differs,
when $\theta \ne 1$, from the intensity $\theta \, dx /x$ of the
corresponding Poisson process.

\begin{corollary}\label{cor PD theta}
 Let $\theta >0$.  View the Poisson-Dirichlet process with parameter
$\theta$, $(X_1,X_2,\dots)$, as the random multisubset of $(0,1]$
given by $A = \{X_1,X_2,\dots \}$. Then the PD is the unique random
$A$ for which both
$$ 
   T := \sum_{x \in A} x \ \  \mbox{ has }  \p(T \le 1 ) = 1,
$$ 
where the sum is taken with multiplicities, 
and for each $k=1,2,\dots$, the multi-intensity density of $A$ is
given by the right side of \eqref{PD theta multi intensity}.
\end{corollary}
\begin{proof}
 If a multiset $A$ satisfies the given hypotheses, then we can apply\break
Lemma~\ref{maintheta} with $A_n=A$,  for each
$n=1,2,\dots$.   Conversely, we have already noted that starting with
the PD, we have $1=\p(T=1)$, and multi-intensity density as given by  
\eqref{PD theta multi intensity}.
\end{proof}

\section{Classic Billingsley}

 We reprove Billingsley's original theorem, even 
though it becomes a special case of a later result.

\begin{theorem}\label{classicB}
 Given $n>1$, let $p_1 \ge p_2 \ge \dots $ be the prime factors, including multiple factors, 
of a random integer $N$ chosen uniformly from $1$ to $n$; and for $i >0$ let
$$
L_i(n) =  \log p_i /\log n,
$$
where we set $L_i(n) = 0$ if $N=1$ or  
$i$ exceeds the total number of prime
factors, including multiplicities.
Define
$$
L(n) = (L_1(n), L_2(n), \dots).
$$
Then $L(n)$ converges to a PD($1$) as $n \to \infty$. 
Equivalently, for each $k>0$ the $k$-tuple $(L_1(n),\dots,L_k(n))$ converges in distribution to 
the first $k$ coordinates of a PD($1$).
\end{theorem}
\begin{proof}
 Define the multiset $A_n$ to contain the non-zero
coordinate entries of $L(n)$, and let $A_n^{\circ}$ be the set underlying $A_n$, i.e.,
with no multiple copies. We want to use Lemma~\ref{maintheta1}. Since 
$|A_n \cap I | \ge |A_n^{\circ} \cap I |$ for any $I = (a,b] \subset (0,1]$ 
and we want lower bounds, it suffices to 
investigate $A_n^{\circ}$.

 Let $I_1 = [a_1,b_1],\dots,I_k = [a_k,b_k]$ be disjoint subintervals of $(0,1]$, 
with $b_1 + \cdots + b_k <1$. The key step is to see that
\begin{equation}\label{divineq}
 \e \prod_{i=1}^k  | A^{\circ}_n \cap I_i| = \sum \frac{1}{q_1\dots q_k} +o(1)
\end{equation}
where the sum is over all $k$-tuples of primes lying in 
$[n^{a_1},n^{b_1}] \times \cdots \times [n^{a_k},n^{b_k}]$.

 To verify \eqref{divineq}, for integers $m,N >0$ let $J_N(m)$ be the
indicator function of the event $m|N$. If $N$ is our random integer we
then have
$$
\prod_{i=1}^k  | A^{\circ}_n \cap I_i| = \prod_{i=1}^k \sum_{q_i} J_N(q_i) = 
\sum \prod_{i=1}^k J_N(q_i) = \sum_{q_1,\dots,q_k} J_N(q_1 \cdots q_k)
$$
where each $q_i$ ranges over the primes in $[n^{a_i},n^{b_i}]$, and where we have used
the disjointness of the intervals in the last  step. Note that since $b:=b_1 + \cdots + b_k <1$
and each $q_i \le n^{b_i}$, each product $q_1\cdots q_k$ in the final sum, and hence also the total
number of summands, cannot exceed $n^b =o(n)$.  

 Since $N$ is uniform random in $[1,n]$
we have 
$$
\p (J_N(m) = 1) =  \frac{1}{n}\lfloor n/m \rfloor =\frac{1}{m} + \frac{1}{n} O(1),
$$
uniformly in $m,n$. Then
$$
\e \prod_{i=1}^k  | A^{\circ}_n \cap I_i| = \sum_{q_1,\dots,q_k} \e J_N(q_1 \cdots q_k)
= \sum \frac{1}{q_1\cdots q_k} +\frac{1}{n}\sum O(1)
$$
$$
 =\sum \frac{1}{q_1\cdots q_k} + O(n^b/n),
$$
establishing \eqref{divineq}.

 Putting it all together  then yields
$$
\liminf_{n \to \infty} \e \prod_{i=1}^k  | A_n \cap I_i|  
\ge \liminf_{n \to \infty} \e \prod_{i=1}^k  | A^{\circ}_n \cap I_i|
= \lim_{n \to \infty} \prod_1^k \sum \frac{1}{q_i} = \prod_1^k \log\frac{b_i}{a_i},
$$ where we have used Mertens' formula \eqref{Mertensoriginal} in the
last step.  
This confirms hypothesis \eqref{logintensineq}.  As for
\eqref{k hyp 2}, since for each $N$ we have $p_1 p_2 \dots = N \le n,$
it is always the case that $T_n = L_1(n) + L_2(n) + \cdots \le 1$.
Therefore, Lemma~\ref{maintheta1} applies, completing the proof.
\end{proof}

\section{New Billingsley for Normed Arithmetic\\ Semigroups}

 In this section we extend Billingsley's theorem to the context of
normed arithmetic semigroups satisfying growth conditions general
enough to allow a PD($\theta$) limit with $\theta \ne 1$.

\subsection{Normed Arithmetic Semigroups}\label{sect knop}

 Following the terminology of \cite{knopfmacher}, a {\em normed
arithmetic semigroup} $\cS$ is a commutative semigroup whose only unit
is $1$, admitting unique factorization into prime
elements,\footnote{i.e., free generators} and equipped with a
nonnegative multiplicative norm function $s \mapsto |s|$ for which any
set of elements with bounded norm is finite. It follows at once that
any element $s \ne 1$ must have norm $|s| >1$.

 Certain growth conditions have been studied, with a view towards
providing abstract settings for classical analytic number theory.  For
$x>0$, let $\nu_{\cS}(x)$ be the number of elements $s \in \cS$ with
$|s| \le x$, and let $\pi_{\cS}(x)$ be the number of primes $\pp \in
\cS$ with $|\pp| \le x$.  The asymptotic linear growth condition, that for some
positive constants $A$ and $ \delta$, we have
\begin{equation}\label{lingrowth}
\nu_{\cS}(x)= Ax\left(1 + O\left(\frac{1}{x^{\delta}}\right)\right), 
\end{equation}
has been studied by, e.g., Knopfmacher in \cite{knopfmacher} (as well as by Beurling before him)
who shows, among many other things,  that given \eqref{lingrowth}
we have a generalized Mertens formula (\cite{knopfmacher}, 
Lemma 2.5) 
asserting that for positive $x$,
\begin{equation}\label{newmert}
\sum_{|\pp| \le x} \frac{1}{|\pp|} = \log \log x + B_{\cS} +O\left(\frac{1}{\log x}\right)
\end{equation}
where the constant $B_{\cS}$ depends on the semigroup. He also proves a prime element theorem,
based on Landau's prime ideal theorem, asserting that for $x >0$
we have
\begin{equation}\label{primelt}
\pi_{\cS}(x)\sim x/\log x,
\end{equation}
though we will not need that here.

 Apart from the ordinary positive integers, of course, the semigroup of integral ideals in a number 
field is a standard example, with growth condition \eqref{lingrowth} derived,
e.g., in \cite{Marcus}, Theorem 39; and many other natural examples are given in \cite{knopfmacher}.
For some additional examples of contemporary interest, 
see \cite{liu2,liu}.

 B. M. Bredikhin has studied normed arithmetic semigroups in which $\pi_{\cS}(x)$ satisfies
$\pi_{\cS}(x)\sim \theta x/\log x$ for fixed $\theta \ne 1$ and has shown, in particular, that if
\begin{equation}\label{primgrowth}
\pi_{\cS}(x) = \theta \frac{ x}{\log x}\left(1+ O\left(\frac{1}{ (\log x)^{\epsilon}}\right)\right)
\end{equation}
for some $\epsilon> 0$, then
\begin{equation}\label{bredgrowth}
\nu_{\cS}(x)=\frac{ Ax}{ \log^{1 - \theta}x} \left(1 + O\left(\frac{1}{ (\log \log)^{\epsilon'}}\right)\right)
\end{equation}
for some positive $A$ depending on $\cS$ and where $\epsilon' = \min(1,\epsilon)$.
See \cite{postnikov}, Section~2.5 for a complete account.

A generalized Mertens formula is given for this case as well, 
in passing, on p. 93 of \cite{postnikov}, namely
$$
\sum_{|\pp| \le x} \frac{1}{|\pp|} = \theta \log \log x + O(1),
$$
but we will need the stronger form
\begin{equation}\label{thetamert}
\sum_{|\pp| \le x} \frac{1}{|\pp|} = \theta \log \log x + B_{\cS} + o(1),
\end{equation}
for some constant $B_{\cS}$ depending only on $\cS$. This follows,
however from \eqref{primgrowth} via a standard Stieltjes integral
argument: Write $G(t) = \pi_{\cS}(t) - \theta t/\log t$, so that $G(t)
= O(t/ (\log t)^{1+ \epsilon})$; also $G$ is clearly of bounded
variation.  Let $r > 1$ be less than the minimum norm value of any
prime element.  Then we have
$$
\sum_{|\pp| \le x} \frac{1}{|\pp|} = \int_r^x \frac{1}{t}\, d \pi_{\cS}(t) = 
\theta \int_r^x \frac{1}{t}\, d\left(\frac{t}{\log t}\right) +
\int_r^x \frac{1}{t}\, dG(t). 
$$
The first integral on the right is 
$$
\log \log x - \log \log r + 1/\log x - 1 /\log r.
$$
As for the second integral, knowing that $G(t) = O(F(t))$, where 
$$
F(t) = t/ (\log t)^{1+ \epsilon}
$$
entitles us to write, via formula (4.67) on p. 57 of \cite{knuthgreene},
$$
\int_r^x \frac{1}{t}\, dG(t) = O\left(\frac{1}{r}F(r)\right) + 
O\left(\frac{1}{x}F(x)\right) + 
O\left( \int_r^x \frac{1}{t} \, dF(t)\right),
$$
an integration by parts trick well-known to analytic number theorists, but not often derived
in textbooks. Thus the integral with respect to $dG(t)$ 
converges as $x \to \infty$, and so we may write
$$
\int_r^x \frac{1}{t} \, dG(t) 
= \int_r^{\infty} \frac{1}{t}\, dG(t) - \int_x^{\infty} \frac{1}{t}\, dG(t).
$$
Collecting terms gives us \eqref{thetamert}, with
$$
B_{\cS} = \int_r^{\infty} \frac{1}{t}\, dG(t) - \theta (\log \log r + 1 /\log r).
$$

 Examples of semigroups $\cS$ satisfying Bredikhin's condition include semigroups of positive
integers all of whose prime factors range among a union of disjoint arithmetic sequences with the
same increment;
by Dirichlet theory, the constant $\theta$ is then the sum of the densities of the primes 
from each sequence, amongst all the primes. 

 For an example with $\theta >1$, let $\cS$ be the commutative
multiplicative semigroup freely generated by two disjoint copies of
the usual primes, where the norm of an element is its ordinary
value. In this $\cS,$ two elements of the same norm are distinct
unless the primes in their respective factorizations can be matched by
type. Then $\cS$ satisfies $\eqref{primgrowth}$ with $\theta =2$, with
remainder term as derived from various versions of the usual prime
number theorem.

\subsection{New Billingsley}

 We will state and prove a version of Billingsley's theorem for normed
arithmetic semigroups $\cS$ of the two types discussed in the
preceding subsection. But first we isolate, as lemmas, two pieces of
an argument that we will use more than once.  We write $\Omega(s)$ for
the number of prime factors of $s$, including multiplicities.

\begin{lemma}\label{abstractsetup}
 Let $\cS$ be a normed arithmetic semigroup. Given $n>1$ let $s$
be chosen according to some probability distribution $\p_n$ from the
elements with norm not exceeding $n$.  If $|s| > 1$ let $s =\pp_1
\pp_2 \dots$ be a decomposition into prime factors, with $|\pp_1| \ge
|\pp_2| \ge \dots$.  For each $n$, define a process $L(n) =
(L_1(n),L_2(n),\dots)$ where each $L_i(n) = 0$ if $s=1$ but otherwise
$$
L_i(n) = \frac{\log|\pp_i|}{\log n}
$$
for $i \le \Omega(s)$,
and 
$$
L_i(n) = 0
$$
for $i > \Omega(s)$. 

Let $A_n$ be the multiset of non-zero elements of $L(n)$, and let 
$I_i = [a_i,b_i] \subset (0,1]$, $i = 1,\dots,k$, be disjoint closed intervals.
Then 
$$
\e \prod_{i=1}^k  | A_n \cap I_i| \ge \sum_{\qq_1, \dots, \qq_k}  \p_n((\qq_1 \cdots \qq_k)|s)
$$
where the primes $\qq_1, \dots, \qq_k$ range over 
the respective 
sets $\{\pp:  |\pp| \in [n^{a_i},n^{b_i}]\}$.

\end{lemma}
\begin{proof}
 For each $n$ let $A^{\circ}_n$ be the set underlying $A_n$, and for $s \in \cS$, 
let $J_s(\cdot)$ be the indicator function for ``divides $s$.''

Given $s$
we have 
\begin{equation}\label{multiformula}
\prod_{i=1}^k  | A^{\circ}_n \cap I_i| = \prod_{i=1}^k \sum_{\qq_i} J_s(\qq_i) = 
\sum \prod_{i=1}^k J_s(\qq_i) = \sum_{\qq_1,\dots,\qq_k} J_s(\qq_1 \cdots \qq_k)
\end{equation}
where each $\qq_i$ runs through the primes with $|\qq_i| \in [n^{a_i},n^{b_i}]$, and we are using 
the disjointness of the intervals in the last equality.
Since $| A_n \cap I| \ge | A^{\circ}_n \cap I|$ for any interval and 
$\e  J_s(\qq_1 \cdots \qq_k) = \p_n((\qq_1 \cdots \qq_k)|s)$,
the result follows.
\end{proof}

 Also, to estimate certain sums 
where the terms are to be approximated, with uniformly small
\emph{relative} error,
we will need the following elementary fact
which takes more space to state than to prove:

\begin{lemma}\label{truly}  

 For $n=1,2,\dots$, let $I_n$ be an arbitrary finite set.  For $i \in
I_n$, let \ $t(i,n),a(i,n) \in \mathbb{R}$, with $a(i,n) \ge 0$, and
let $T_n := \sum_{i \in I_n} t(i,n)$ and $A_n := \sum_{i \in I_n}
a(i,n)$.  Assume that $c :=\lim_n A_n$ exists and $c \ge 0$.  Assume
that for all $n$, for all $i \in I_n$,
$$
  t(i,n) = a(i,n) ( 1+ e(i,n)),
$$
with 
$$ E_n := \sup_{i \in I_n} | e(i,n)| \mbox{ satisfying } E_n \to 0.
$$
Then $T_n \to c$.
\end{lemma}
\begin{proof} 
\begin{align*}
  | T_n - A_n | \le& \sum_{i \in I_n} |t(i,n) - a(i,n)| 
 = \sum |e(i,n)| \ a(i,n)\\
\le& \sum E_n \ a(i,n) = E_n A_n \to 0 \times c = 0.\\
\end{align*}
\end{proof}

 We now prove the main theorem of this section.
\begin{theorem}\label{abstractBill}
 Let $\cS$ be a normed arithmetic semigroup satisfying either
\eqref{lingrowth} or \eqref{primgrowth}.  Given $n>1$ let $s$ be
chosen uniformly from the elements with norm not exceeding $n$.  Let
$s =\pp_1 \pp_2 \dots$ be a decomposition into prime factors, with
$|\pp_1| \ge |\pp_2| \ge \dots$.  For each $n$, define a process $L(n)
= (L_1(n),L_2(n),\dots)$ where each $L_i(n) = 0$ if $s=1$ but
otherwise,
$$
L_i(n) = \frac{\log|\pp_i|}{\log n}
$$
for $i \le \Omega(s)$,
and 
$$
L_i(n) = 0
$$
for $i > \Omega(s)$. Then as $n \to \infty$ 
$$
(L_1(n),L_2(n),\dots) \Rightarrow (L_1,L_2,\dots)_{\theta},
$$
a PD($\theta$) limit 
where $\theta = 1$ if \eqref{lingrowth} holds, while otherwise $\theta$ takes the same value 
as in \eqref{primgrowth} if that formula holds.  
\end{theorem}
\begin{proof}
 We have $0\le L_i(n) \le 1$ for each such term, 
and $T_n := \sum_i L_i(n) = \log|s| /\log n \le 1$,
so hypothesis \eqref{theta k hyp 2} of Lemma~\ref{maintheta} is satisfied.

 As for hypothesis \eqref{theta intensineq}, 
given $k >0$ let $I_i = [a_i,b_i]$, $i =1,\dots,k$ be disjoint  closed subintervals
of $(0,1]$ with $b_1 + \cdots + b_k <1$.

 Retaining the notation of Lemma~\ref{abstractsetup}, 
if $s$ is chosen uniformly from the elements with $|s| \le n$, then since
$$
\p_n((\qq_1 \cdots \qq_k)|s) = \frac{\nu_{\cS}(n/|\qq_1 \cdots \qq_k|)}{\nu_{\cS}(n)},
$$
it will suffice by that lemma to investigate $\lim_{n \to \infty}$ (or at least $\liminf$) of
\begin{equation}\label{Eformula}
\sum_{\qq_1,\dots,\qq_k} \frac{\nu_{\cS}(n/|\qq_1 \cdots \qq_k|)}{\nu_{\cS}(n)}
\end{equation}
where each $\qq_i$ runs through the primes with $|\qq_i| \in [n^{a_i},n^{b_i}]$.

 Write $b:= b_1 + \cdots + b_k <1$. If \eqref{lingrowth} is in effect, then  
\eqref{Eformula} becomes
\begin{equation}
\sum \frac{1}{|\qq_1|\cdots|\qq_k|}(1 + e(|\qq_1|\cdots|\qq_k|,n))
\end{equation}
where the condition $b <1$ ensures that $e(|\qq_1|\cdots|\qq_k|,n) =o(1)$
as $n \to \infty$, uniformly over choices of $\qq_1,\dots,\qq_k$.
By Lemma~\ref{truly} coupled with the generalized Mertens formula \eqref{newmert} 
we thus have 
$$
\lim_{n \to \infty} \sum_{\qq_1,\dots,\qq_k}
\p_n((\qq_1 \cdots \qq_k)|s) = \lim_{n \to \infty} \sum \frac{1}{|\qq_1|\cdots|\qq_k|}
$$
$$
= \lim_{n \to \infty} \prod_{i=1}^k \sum_{\qq_i \in [n^{a_i},n^{b_i}]}\frac{1}{|\qq_i|}
= \prod_{i=1}^k \log(\frac{b_i}{a_i}).
$$
Thus by Lemma~\ref{abstractsetup}, together with
Lemma~\ref{sectheta1} with $\theta = 1$ 
(or Lemma~\ref{maintheta1}, for that matter), we are done when \eqref{lingrowth} is in effect.

 If Equation~\eqref{primgrowth} and hence \eqref{bredgrowth} are in effect instead, 
then 
\eqref{Eformula} becomes
$$
\sum \frac{1}{|\qq_1|\cdots|\qq_k|} 
\left(\frac{1}{1 - \log_n |\qq_1| - \cdots - \log_n |\qq_k|}\right)^{1 - \theta}
(1 + e(|\qq_1|\cdots|\qq_k|,n))
$$
where we use $n/|\qq_1| \cdots |\qq_k| \ge n^{1-b}$ to ensure that 
$e(|\qq_1|\cdots|\qq_k|,n) =o(1)$
as $n \to \infty$, uniformly over choices of $\qq_1,\dots,\qq_k$.

 Also since, writing $a: = a_1 + \cdots + a_k$ we have 
$$
\frac{1}{1 - \log |\qq_1|/\log n - \cdots - \log |\qq_k|/\log n}
\ge \frac{1}{(1 -a)},
$$
we find again from Lemma~\ref{truly} that 
\begin{multline}
\liminf \sum_{\qq_1,\dots,\qq_k} \frac{\nu_{\cS}(n/|\qq_1 \cdots \qq_k|)}{\nu_{\cS}(n)} 
\ge 
\frac{\theta^k}{(1-a)^{1-\theta}}\prod_{i=1}^k
\log\left(\frac{b_i}{a_i}\right), 
\end{multline}
this time using the Mertens formula \eqref{thetamert}.

 Once again we have satisfied hypothesis \eqref{logtheta intensineq} of Lemma  \ref{maintheta}, 
now with $\alpha = 1-\theta$ and $\beta = 0$.
This completes the proof.
\end{proof}

\subsection{A Pair of Examples, $\theta=1/2$ and $\theta=1$}

 The positive integers representable as sums of two squares form a normed arithmetic 
semigroup $\cS$ whose generating primes $\pp$ consist of the prime $\pp = 2$, the primes
$\pp = p \equiv 1$ mod($4$), and the  ``square primes'' 
$\pp = p^2$ where $p \equiv 3$ mod($4$). See standard texts for this
theory. We take  
$|\pp| = \pp$, of course.
>From Dirichlet theory we know that
\begin{equation}\label{sumsqrowth}
\pi_{\cS}(x) = \frac{1}{2} \frac{x}{\log x} 
\left(1 + O\left(\frac{1}{\log x}\right)\right),
\end{equation}
which is \eqref{primgrowth} with $\theta=1/2$, $\epsilon=1$.  Therefore, Theorem \ref{abstractBill} applies, 
giving us a limiting PD($\theta$) with $\theta = 1/2$.

 Also, the Gaussian integers $\{r + is: r,s \in \mathbb{Z} \}$ 
form a principal ideal domain, with unique factorization
up to multiplicative units.   Mod  out by the unit group $\{ \pm 1, \pm i \}$, to get a normed 
arithmetic semigroup\footnote{ Defined in Section~\ref{sect knop}.}
$\cS$ with 
norm $|r+si| = r^2 + s^2$. This semigroup satisfies
\begin{equation}
\nu_{\cS}(x) = \frac{\pi}{4}x\left(1 +O\left(\frac{1}{\sqrt{x}}\right)\right),
\end{equation}
which is \eqref{lingrowth} with $A=\pi/4$ and $\delta=1/2$.
Therefore, Theorem \ref{abstractBill} applies here as well, but with a limiting 
PD($\theta$) having 
$\theta = 1$.

 It is well-known, however, that the positive integers appearing as
norms of primes in the two cases are identical. Therefore, the numbers
appearing in the respective sequences
$$
L(n) = (L_1(n),L_2(n),\dots)
$$ are also identical; but we get different limiting behaviors. One
could unify this pair of examples by saying that in {\em both} cases
we are actually selecting random $N = r^2 + s^2$ with $N \le n$, i.e.,
from the first semigroup; but in the first case the selection is
uniform, while in the second case the probabilities are proportional
to the number of representations as sums of two squares (or as norms
of Gaussian integers).

\section{Integers with Unusual Numbers of Prime Factors}

 In this final section we derive another extension of Billingsley's theorem,  for a situation
that does not seem to be covered by Theorem \ref{abstractBill}.

\subsection{The Tur{\'a}n and Erd\H{o}s-Kac Theorems and Selberg's Formulas}

 Given a positive integer $N$, let $\om(N)$ denote the number of
distinct prime factors of $N$ and let $\Om(N)$ denote the number of
prime factors counted with multiplicities. The Erd\H{o}s-Kac theorem
asserts that if $N$ is picked uniformly at random from $1$ to $n$,
then as $n \to \infty$, the quantities
$$
\frac{ \Om(N)- \log \log n}{\sqrt{\log \log n}}
$$ 
and 
$$
\frac{ \om(N)- \log \log n}{\sqrt{\log \log n}}
$$
converge in distribution to standard Gaussian variables. Furthermore, Turan's
theorem from 1934 gives an asymptotic bound for the probability of certain
large deviation events.   
Namely, if $\xi(n) \to \infty$ with $n$, then the 
probability that the absolute value of either quantity exceeds $\xi(n)$ 
is $O(1/\xi^2)$. (See,
e.g., \cite{TenenbaumEnglish},
Section III.3.) So asymptotically, the events 
$$
|\Om(N)- \log \log n| \ge \epsilon \log \log n
$$
and
$$
|\om(N)- \log \log n| \ge \epsilon \log \log n,
$$
for any fixed $\epsilon >0$, become vanishingly rare,
and one would expect such integers $N$
to be atypical in the distribution of their large prime factors, as well as in the number
of prime factors. 

 In the next section we will consider random integers $N$  picked uniformly 
from those positive integers not exceeding $n$ and for which either $\Om(N)$ or $\om(N)$
is required, roughly speaking, to stay vanishingly close to 
$\tau   \log \log n$ for some $\tau >0$. We will show that
a version of Billingsley's theorem is once again valid, 
with a PD($\theta$) limit as $n \to \infty$,
where $\theta$ is sometimes, but not always, equal to $\tau$. 

 The proof will exploit three growth formulas due to Selberg and Delange, which we record here:
Write
$$ 
  \nu_j(x) := \{ m \le x: \Omega(m) =j \}
$$ 
for the count of positive integers, not exceeding $x$, and having
exactly $j$ prime factors including multiplicity.
Theorems II.6.5 and 6 in \cite{TenenbaumEnglish}
describe the growth of these counts involving $\Omega(N)$, as follows. 
Given $\delta >0,$ we have 
uniformly over $x \ge 3$ and $1 \le j \le $ \mbox{$(2-\delta) \log \log x$},

\begin{equation}\label{eqn 9}
\nu_j(x)
=\frac{x}{\log x}\cdot \frac{(\log \log x)^{j-1}}{(j-1)!}\left\{\kappa 
\left(\frac{j-1}{\log \log x}\right)
+O\left(\frac{j}{\left(\log \log x\right)^2}\right)\right\}
\end{equation}
where $\kappa$ is 
\[
\kappa(z) = \frac{1}{\Gamma(z+1)} 
\prod_p \left( 1-\frac{z}{p} \right)^{-1}\left(1-\frac{1}{p} \right)^z,
\]
with product taken over the primes.
Note that $\kappa$ is continuous, apart from poles at $2,3,\dots$.
Also for $\delta >0$,
uniformly over $x \ge 3$ but, this time,  $j/\log \log x \in (2+\delta,A)$ for any fixed 
$A \in (2+\delta,\infty)$, we have
\begin{equation}\label{eqn 10}
\nu_j(x)=C\frac{x \log x}{2^j}
\left\{   1+O\left(\frac{1}{(\log x)^{\delta^2/5}}\right)\right\}
\end{equation}
where $C$ is the 
constant $C := (1/4) \prod_{p>2} (1+1/(p(p-2))) \doteq 0.378694$.

 Here is the corresponding formula when $\Omega$ is replaced with $\omega$.
Write
$$ 
  \nu^{\circ}_j(x) := \{ m \le x: \omega(m) =j \}
$$ 
for the count of positive integers, not exceeding $x$, and having
exactly $j$ distinct prime factors, not counting multiplicity.
Theorem II.6.4 in \cite{TenenbaumEnglish}
 states that given $A>0,$ we have uniformly over $x \ge 3$ and
$1 \le j \le A\log \log x$,
\begin{equation}\label{eqn 11}
\nu^{\circ}_j(x)
=\frac{x}{\log x}\cdot \frac{(\log \log x)^{j-1}}{(j-1)!}\left\{\lambda 
\left(\frac{j-1}{\log \log x}\right)
+O\left(\frac{j}{\left(\log \log x\right)^2}\right)\right\}
\end{equation}
where $\lambda$ is 
\[
\lambda(z) = \frac{1}{\Gamma(z+1)} 
\prod_p \left( 1+\frac{z}{p-1} \right)\left(1-\frac{1}{p} \right)^z.
\]
Note that $\lambda$ is continuous at all non-negative values of the argument $z$.

 For later use we state an approximation involving the leading factors
of~\eqref{eqn 11}. 
\begin{lemma}\label{approxsel}
 Given $\theta > 0$,  let $k= k(x)$ be nonnegative integers with $k \sim \theta \log \log x$ 
as $x \to \infty$. Then 
$$
\frac{x}{\log x} \frac{(\log \log x)^k}{k!}
 \sim \frac{x}{ (\log x)^{\theta'}\sqrt{2 \pi \theta \log \log x}}
$$
where $\theta' = 1- \theta(1-\log \theta) \ge 0$.
\end{lemma}
\begin{proof}
 This follows from Stirling's formula.
\end{proof}

\noindent{\em Remark}.  
 Note that
$$
\frac{1}{\log x} \frac{(\log \log x)^k}{k!}$$ 
is the probability of the integer $k$ for a Poisson distribution
having mean $\mu = \log \log x$.  Thus Selberg's formula asserts, in
particular, that the number of distinct prime factors of a random
integer picked uniformly from $1$ to $n$ is distributed,
asymptotically, like $Z-1$ where $Z$ is a Poisson random variable with
mean $\log \log n$, conditional on $Z \ge 1$.

 Also, Reference~\cite{liu2} proves an extension of the
Erd\H{o}s-Kac theorem to factorizations in normed arithmetic
semigroups satisfying \eqref{lingrowth}.  This raises the possibility
that Selberg's formulas \eqref{eqn 9}, \eqref{eqn 10}, and \eqref{eqn
11} might also extend to that context, which would then lead to
similar extensions of the theorems proved below.

\subsection{Generalized Billingsley for Unusual $\Om(N)$}

\begin{theorem}\label{OMEGA theorem}
 Fix $\tau \ne 2 \in (0,\infty)$, and let 
$g(n)$
be any sequence of integers, with $1 \le g(n) \le \log_2(n)$, such that as $n \to \infty$
\begin{equation}\label{unusual g}
  \frac{g(n)}{\log \log n} \to \tau.
\end{equation}
Pick $N$ uniformly from the set of positive integers $\{m: \ m \le n,
\Omega(m)=g(n)\}$, and let $p_1 \ge p_2 \ge \dots$ be the sequence of prime factors, including
multiplicities.
Define $L(n) = (L_1(n),L_2(n),\dots)$ where $L_i(n) = \log p_i /\log n$
for $i \le
g(n)$ and $L_i(n)=0$ for $i > g(n)$.
Then, with 
$$ 
\theta := \min(\tau,2),
$$
we have convergence in distribution to the Poisson-Dirichlet with
parameter $\theta$:
$$
  L(n) \Rightarrow PD(\theta).
$$
\end{theorem}

\begin{proof}

 We retain the multiset notation of Theorems \ref{maintheta} and \ref{abstractBill}.

 To show that the hypothesis \eqref{theta intensineq} 
is satisfied we observe 
that if $\p_n$ is the measure where $N$ is picked uniformly 
from $\{m: \ m \le n, \Omega(m)=g(n)\}$, then Lemma
\ref{abstractsetup} applies, so that for
fixed disjoint subintervals $[a_i,b_i]$, with $0 < a := a_1+\cdots+a_k
< b:= b_1+\cdots+b_k<1$, and
with $p_i$ ranging over $[n^{a_i},n^{b_i}]$, we have 
\begin{equation}\label{expineq}
\e \prod_{i=1}^k  | A_n \cap I_i| \ge \sum_{ p_1, p_2, \dots, p_k} \p_n( p_1 p_2 \dots p_k | N ).
\end{equation}

 Next, we will establish inequalities of the following sort: 
\begin{equation}\label{Omega model goal}
  \p_n( p_1 p_2 \dots p_k | N ) \ge \frac{1}{p_1p_2\cdots p_k} \ 
\frac{1}{(1-a)^{\alpha}(1-a)^{\beta}} \ \theta^k \left( 1 + o(1) \right)
\end{equation}
as $n \to \infty$, with error term uniform over choices of $p_1,p_2,\cdots, p_k$,
with $\alpha + \beta = 1-\theta$.
This will lead to the required lower bound for
$$ \liminf\e \prod_{i=1}^k  | A_n \cap I_i|.$$

 Now, specifying that we are to pick $N \le n$ uniformly from the integers with exactly
$g(n)$ prime factors, including multiplicity, means that we can write 
\begin{equation}\label{Omega fraction}
 \p_n( p_1 p_2 \dots p_k | N )  = 
  \frac{\nu_{g(n)-k}(n/(p_1\cdots p_k))}{\nu_{g(n)}(n)}.
\end{equation}

 For the case $\tau < 2$, we apply Selberg's approximation 
\eqref{eqn 9} to both the numerator and denominator of the 
right side of \eqref{Omega fraction}, switching in the numerator from $n$ 
to $n/(p_1\cdots p_k)$ in the role of $x$, and from $g(n)$ to $g(n)-k$
in the role of $j$.
The first factor $x$ on the right side of \eqref{eqn 9} yields, as a factor on the right side of 
\eqref{Omega fraction},   the ratio of  $n/(p_1\cdots p_k)$ to $n$, 
which is exactly exactly $1/(p_1\cdots p_k)$.  
>From the next factor,
$1/\log x$, we get the ratio
$$
\frac{\log n}{\log(n/(p_1\cdots p_k))} \ge 1/(1-a).
$$

 From the next factor, $(\log \log x)^{j-1}/(j-1)!$,
using $g(n)/\log \log n \to \tau$, we get a ratio $\tau^k (1+o(1))$
which comes from changing the power of $\log \log x$ by $k$, along with changing
the base of the factorial by $k$; 
and we get an additional factor due to changing
the $x$ inside $(\log \log x)^j$, namely

\begin{multline}
  \left(\frac{ \log \log (n/(p_1\cdots p_k))}{ \log \log n}\right)^g
\ge  \left(\frac{\log\log n^{1-b}}{ \log \log n}\right)^g
\\ \phantom{lsfslfjlsjfsdddddddddddddddddddl} = 
\left( 1 + \frac{\log (1-b)}{\log\log n} \right)^g =
(1-b)^\tau (1+o(1)).\nonumber \\
\end{multline}

 The last factor of \eqref{eqn 9}, shown in large braces, contributes
another
$1+o(1)$ to the product of ratios:  one must check that $\kappa(z)$,
 evaluated at the two specified arguments, yields the claimed asymptotic ratio, and it is
 here that the continuity of $\kappa$ is invoked.

 The net result is that for the right side of \eqref{Omega fraction}
we have
$$
 \frac{\nu_{g(n)-k}(n/(p_1\cdots p_k))}{\nu_{g(n)}(n)} \ge
\frac{1}{p_1\cdots p_k} \ \frac{1}{1-a} \ \tau^k (1-b)^\tau (1+o(1)),
$$ with the $o(1)$ relative error term uniformly small over choices of
$p_1,\dots,p_k$.  Summing both sides over $p_i \in [n^{a_i},n^{b_i}]$
and appealing to the classical Mertens' formula
\eqref{Mertensoriginal} as well as Lemma~\ref{truly}, coupled with
\eqref{expineq} and \eqref{Omega fraction}, we see that~\eqref{theta
intensineq} is satisfied with $\alpha = 1$ and $\beta = -\tau$.  Also
\eqref{theta k hyp 2} is trivially confirmed.  Therefore, Theorem
\ref{maintheta} applies, completing the argument for $\tau < 2$.

 For the case with $\tau >2$, we proceed as above, but substituting the use of the much
simpler \eqref{eqn 10} for 
\eqref{eqn 9} in deriving a lower bound for \eqref{Omega fraction}. 
Again we get a product of ratios: The first factor, $x/2^j$, 
yields the ratio $2^k/(p_1\cdots p_k)$.
The next factor, $\log x$, yields the ratio 
$$
\frac{\log(n/(p_1\cdots p_k))}{\log n} \ge 1-b.
$$
The final factor yields a ratio which is $1+o(1)$ as $n \to \infty$,  
uniformly over choices of $p_1,\dots,p_k$.  Once again, summing over $p_i \in [n^{a_i},n^{b_i}]$ 
and applying Mertens shows that Theorem~\ref{maintheta} applies,
with $\alpha = 0$ and $\beta = -1$, leading to a PD(2) limit.
\end{proof}

\subsection{Generalized Billingsley for Unusual $\omega(N)$}

\begin{theorem}\label{omega theorem}
 Fix $\theta \in (0,\infty)$, and let 
$g(n)$
be any sequence of integers, with $1 \le g(n) \le \log_2(n)$, such that as $n \to \infty$
\begin{equation}\label{unusual gomega}
  \frac{g(n)}{\log \log n} \to \theta.
\end{equation}
Pick $N$ uniformly from the set of positive integers $\{m: \ m \le n,
\omega(m)=g(n)\}$, and let $p_1 > p_2 > \dots$ be the sequence of distinct prime factors, 
i.e., not including multiplicities.
Define $L(n) = (L_1(n),L_2(n),\dots)$ where $L_i(n) = \log p_i /\log n$
for $i \le
g(n)$ and $L_i(n)=0$ for $i > g(n)$.
Then, 
we have convergence in distribution to the Poisson-Dirichlet with
parameter $\theta$:
$$
  L(n) \Rightarrow PD(\theta).
$$
\end{theorem}
\begin{proof}
 The plan of proof is similar to that of the first part of Theorem
\ref{OMEGA theorem}, using \eqref{eqn 11} in place of \eqref{eqn 9},
but there is a new technical issue to confront: if $p_1\cdots p_k|N$
then we have $\Omega(N/(p_1\cdots p_k)) = \Omega(N) -k$ always, a fact
exploited in writing \eqref{Omega fraction}. However,
$\omega(N/(p_1\cdots p_k))$ may lie anywhere from $\omega(N) -k$ to
$\omega(N)$, depending on the multiplicities of the $p_i$ in $N$, a
fact which complicates the adaptation of \eqref{Omega fraction} to
$\omega$.

 Fortunately, large prime factors with multiplicities greater
than $1$ are sufficiently rare that \eqref{Omega fraction} can  still be used with $\nu^{\circ}$
in place of $\nu$, after
an asymptotically negligible tweak. 
Define the set of positive integers
$$
S= \{ m \le n/(p_1\cdots p_k): \omega(m) = g(n) - k, \mbox{ and } 
p|m \Rightarrow p \not\in \{p_1,\dots,p_k\} \}.
$$
Then via multiplication by $p_1 \cdots p_k$, S maps injectively to a subset of 
$$
T = \{ m \le n :\omega(m) = g(n), \mbox{ and } (p_1\cdots p_k)|m \},
$$
where $T$ is the set whose (relative) cardinality we wish to bound fairly sharply from below.
Also we have
\begin{align*}|S| & \ge \nu^{\circ}_{g(n)-k}(n/(p_1\cdots p_k)) -
  \sum_{p_i} (n/(p_1\cdots p_k))/p_i \\
&\ge  \nu^{\circ}_{g(n)-k}(n/(p_1\cdots p_k)) -  kn^{1-a_1}/(p_1\cdots
  p_k), 
\end{align*}
and so
$$
\p_n(p_1\cdots p_k|N) = \frac{|T|}{ \nu^{\circ}_{g(n)}(n)} \ge 
\frac{ \nu^{\circ}_{g(n)-k}(n/(p_1\cdots p_k))}{ \nu^{\circ}_{g(n)}(n)}
- \frac {  kn^{1-a_1}/(p_1\cdots p_k)}{{ \nu^{\circ}_{g(n)}}(n)}.
$$
$$
= \frac{ \nu^{\circ}_{g(n)-k}(n/(p_1\cdots p_k))}{ \nu^{\circ}_{g(n)}(n)}(1+o(1))
$$
where the $o(1)$ is uniform over choices of $p_1,\dots,p_k$ as $n \to
\infty$, 
using Lemma~\ref{approxsel}.

 Also comparison of \eqref{eqn 9} with \eqref{eqn 11} makes it plain that 
exactly as for the right-hand side of \eqref{Omega fraction}, we get
$$
 \frac{\nu^{\circ}_{g(n)-k}(n/(p_1\cdots p_k))}{\nu^{\circ}_{g(n)}(n)} \ge
\frac{1}{p_1\cdots p_k} \ \frac{1}{1-a} \ \theta^k (1-b)^\theta (1+o(1)),
$$ 
with the $o(1)$ relative error term uniform over choices of $p_1,\dots,p_k$.
  So once again we get a PD($\theta$) limit,  as claimed.
\end{proof}

\bibliography{extensions}
\bibliographystyle{plain}

\section*{Keywords}
\raggedright\tt 
Poisson, Dirichlet, weak convergence, Billingsley, Turan, number field
\end{document}